\documentclass[11pt]{amsart}
\usepackage[margin={1in}]{geometry}
\usepackage{fourier}
\usepackage{color}
\definecolor{c0}{rgb}{0.0,0.9,0.9}
\definecolor{c1}{rgb}{0.0,0.5,0.6}
\definecolor{c2}{rgb}{0.0,0.0,0.5}
\definecolor{c3}{rgb}{0.5,0.0,0.5}
\definecolor{grn}{rgb}{0,0.4,0}
\definecolor{dgrn}{rgb}{0.0,0.5,0.0}
\definecolor{dblu}{rgb}{0.0,0.1,0.5}
\definecolor{dpur}{rgb}{0.7,0.0,0.7}
\definecolor{dpurb}{rgb}{0.3,0.0,0.4}
\usepackage[colorlinks=true, linkcolor=c1, citecolor=c3]{hyperref}
\usepackage{amsmath,amssymb,amsthm}
\usepackage{comment}
\newtheorem{innercustomthm}{\textcolor{dpurb}{\textbf{Theorem}}}
\newenvironment{customthm}[1]
{\renewcommand\theinnercustomthm{#1}\innercustomthm}
  {\endinnercustomthm}

\numberwithin{equation}{section}
\theoremstyle{definition}
\newtheorem{prop}{\textcolor{c1}{Proposition}}
\newtheorem{lemma}[prop]{\textcolor{c1}{Lemma}}

\newtheorem{cor}[prop]{\textcolor{c1}{Corollary}}
\numberwithin{prop}{section}
\newtheorem{defn}[prop]{\textcolor{c1}{Definition}}

\newtheorem{rmk}[prop]{\textcolor{c1}{Remark}}
\newcommand{\del}{\partial}
\newcommand{\brs}[1]{\left| #1 \right|}
\newcommand{\brk}[1]{\left[ #1 \right]}
\newcommand{\prs}[1]{\left( #1 \right)}

\newcommand{\ip}[1]{\left<#1\right>}
\newcommand{\gG}{\Gamma}
\newcommand{\gY}{\Upsilon}

\newcommand{\gs}{\sigma}

\newcommand{\gw}{\omega}

\newcommand{\N}{\nabla}
\newcommand{\hsp}{\hspace{0.5cm}}

\DeclareMathOperator{\Rc}{Rc}
\DeclareMathOperator{\Rm}{Rm}
\DeclareMathOperator{\End}{End}
\DeclareMathOperator{\cP}{P}
\DeclareMathOperator{\Id}{Id}

\DeclareMathOperator{\cV}{V}

\begin{document}
\title[Almost Hermitian Ricci flow]{Almost Hermitian Ricci flow}
\date{\today}
\author{Casey Lynn Kelleher}
\address{702 Fine Hall\\
         Princeton University\\
         Princeton, NJ 08540}
         \author{Gang Tian}
\address{BICMR and SMS \\ Peking University \\ Beijing, P.R.China, 100871
}
\begin{abstract} We introduce a new curvature flow which
matches with the Ricci flow on metrics and preserves the almost
Hermitian condition. This enables us to use Ricci flow to study almost
Hermitian manifolds.
\end{abstract}
\thanks{The first author was supported by a National Science Foundation postdoctoral fellowship and the second author is partially supported by NSFC grant 11890660.}
\maketitle
\section{Introduction}
The Ricci flow, introduced by R. Hamilton in \cite{Hamilton1}, has proven to be a very useful tool in geometry and topology. Evidence of its utility can be seen through its application in Perelman's solutions of the Poincar\'{e} conjecture and the Geometrization conjecture for $3$-manifolds \cite{Perelman, Perelman1}. To exploit the power of this analytic tool, it is advantageous to identify settings where the flow respects and interacts well with the present geometric structures. For instance, the Ricci flow has been applied to manifolds with various positivity conditions on curvature (eg. manifolds with $1/4$-pinch curvature, \cite{AN, BS, BS1}). Another natural setting is on K\"{a}hler manifolds, where the Ricci flow has been applied to obtain numerous results. Given this, the Ricci flow proves a promising tool for the Analytic Minimal Model Program in the mission to classify compact K\"ahler manifolds birationally \cite{SongT}.

The use of the Ricci flow on K\"{a}hler manifolds so far indicates that, to some degree, the Ricci flow respects and ties in with the delicate interplay of a Riemannian metric and almost complex structure. How far does this utility go? Over the years it has been generally believed that the Ricci flow does not preserve non-K\"{a}hler Hermitian structures. Throughout the literature there are examples of flows on non-K\"{a}hler manifolds (e.g. almost K\"{a}hler \cite{ST} and Hermitian \cite{ST1, Gill}) which apply techniques inspired by the Ricci flow to complex manifolds. In these settings, the metric flow must be modified to flow by some type of curvature quantity rather than the Ricci tensor. This modification helps ensure the flow preserves whatever desired structure is at play. As such, it was yet unclear to what extent the Ricci flow itself could be employed in such settings.

In contrast to the highly structured K\"{a}hler manifold, an almost Hermitian manifold is simply an almost complex manifold with compatibility imposed between the metric and complex structure. In this short paper, we demonstrate that given an almost complex manifold, by coupling with a flow on almost complex structures, the Ricci flow in fact \emph{does} preserve this compatibility. This opens doors to applying the techniques of Ricci flow to studying properties of almost Hermitian structures with an eye towards identifying new classes of manifolds.
\subsection{Outline of paper and statement of main results}
Given an almost Hermitian manifold $(M,g,J)$ with Riemannian metric $g$ and almost complex structure $J$, we denote by $\omega$ its K\"ahler form and by $\nabla$ the associated Chern connection.
We will define the \emph{Almost Hermitian Ricci flow} to be a paired flow of the metric and almost complex structure given by
\begin{align}
\begin{split}\label{eq:AHCF}
\begin{cases}
\prs{\tfrac{\del g}{\del t}}_{bc} &= - 2 \Rc_{bc} \ \\
\prs{\tfrac{\del J}{\del t}}_{a}^c &= - \prs{\cP_{av} - 2 J_a^y \Rc_{yv} }^{2,0+0,2} g^{vc} 
-\prs{\prs{\mathcal{L}_{\vartheta} J}^m_a g_{mv}}_{\text{skew}} g^{vc} +\prs{\kappa_{av} + 2 \N^e  \prs{d \gw}_{eav}^{3,0+0,3}}g^{vc} . \end{cases}
\end{split}
\end{align}
Here, $\cP$ denotes the Chern Ricci tensor, $\Rc$ denotes the Ricci tensor, and $\kappa$ is a lower order $J$-antiinvariant skew symmetric two cotensor.
The first collection of terms in the evolution of the almost complex structure is in fact the main operator in the symplectic curvature flow setting (cf. \cite{ST} pp.182). When viewed in the almost Hermitian setting (rather than almost K\"{a}hler), the first collection of terms may not assure that \eqref{eq:AHCF} is parabolic modulo diffeomorphisms except in dimension $4$. This is why we need to modify by the addition of an extra term in order to create a flow in arbitrary dimension with desired behaviour. Note the resultant flow of the symplectic form is given by
\begin{align*}
\prs{\tfrac{\del \gw}{\del t}}_{ab} &= - \cP_{ab} + \underset{\mathcal{C}}{\underbrace{\prs{\cP^{1,1} - 2 J_a^c \Rc^{1,1}_{bc}} -\prs{\prs{\mathcal{L}_{\vartheta} J}^m_a g_{mb}}_{\text{skew}} + \kappa_{ab} + 2 \N^e  \prs{d \gw}_{eab}^{3,0+0,3}}}.
\end{align*}
Note in particular that if we define
\begin{align*}
\widetilde{\omega}_t \triangleq \omega_t - \int_0^t \mathcal{C}(s)\, ds,
\end{align*}
it follows that $\tfrac{\del \widetilde{\omega}}{\del t} = - \cP$, and thus $\widetilde{\omega}$ stays closed along the flow.
Here is our main result.
\begin{customthm}{A}
For $(M^n, g, J)$, the almost Hermitian Ricci flow is a parabolic flow modulo diffeomorphisms
which preserves the almost Hermitian condition.
\end{customthm}
\section{Background on almost Hermitian manifolds}
In this section we describe baseline concepts for almost Hermitian manifolds. In the first subsection we discuss key objects, while in the following subsection we use these objects to construct computational quantities and analyse their behaviours.
\subsection{Basics}
We begin with a Riemannian manifold $(M,g)$ equipped with an almost complex structure $J$ (that is, $J \in \End (TM)$ such that $J^2 = -\Id$). An almost Hermitian manifold satisfies $(M,g,J)$ satisfies the additional condition that $J$ is $g$-compatible, namely that
\begin{align}
\begin{split}\label{eq:gcompwdefn}
g \prs{X,Y} = g \prs{JX,JY} \qquad & g_{\mathbf{ij}} = J_{\mathbf{i}}^a J_{\mathbf{j}}^bg_{ab}.
\end{split}
\end{align}
From $g$ and $J$ arise a natural symplectic form $\gw$ given by
\begin{align}
\begin{split}\label{eq:gcompwdefn}
\gw \prs{X,Y} \triangleq g\prs{JX,Y } \qquad &\gw_{\mathbf{ij}} = J_\mathbf{i}^a g_{\mathbf{j}a}.
\end{split}
\end{align}%

In coordinates, resultant identities are as follows:\footnote{To clarify, we denote by $\omega^{i j}$ those entries of the inverse $\omega^{-1}$ instead of the raising of $\omega$ by the metric $g$.}
\begin{align*}
g_{\mathbf{ij}} = J^s_{\mathbf{i}} \gw_{\mathbf{j}s}& \qquad g^{\mathbf{ij}} = J_s^{\mathbf{i}} \gw^{\mathbf{j}s} \qquad \gw^{\mathbf{ij}} = J^{\mathbf{i}}_s g^{\mathbf{j}s}, \qquad J_{\mathbf{i}}^{\mathbf{j}} = \gw^{\mathbf{j}s} g_{\mathbf{i}s} = \gw_{\mathbf{i}s} g^{\mathbf{j}s}, \qquad \gw^{\mathbf{a}c} \gw_{c \mathbf{b}} = \delta^{\mathbf{a}}_{\mathbf{b}}.
\end{align*}
We have the following decomposition of $TM_{\otimes 2}$,
\begin{align*}
A_{\mathbf{ij}}^{1,1} &= \tfrac{1}{2}\prs{A_{\mathbf{ij}} + A_{ab} J_{\mathbf{i}}^a J_j^b }, \qquad A_{\mathbf{ij}}^{2,0+0,2} = \tfrac{1}{2}\prs{A_{\mathbf{ij}} - A_{ab} J_{\mathbf{i}}^a J_{\mathbf{j}}^b }.
\end{align*}
Likewise for $TM_{\otimes 3}$, we have subsets of $TM_{\otimes 3}$ given by $TM_{\otimes 3}^{2,1+1,2}$ and $TM_{\otimes 3}^{3,0+0,3}$ with projections
\begin{align}
\begin{split}\label{eq:3formtypedecomp}
B^{2,1+1,2}_{\mathbf{ijk}} &\equiv\tfrac{3}{4}B_{\mathbf{ijk}} + \tfrac{1}{4} \prs{J_j^b J_k^c B_{ibc} + J_{\mathbf{i}}^a J_{\mathbf{k}}^c B_{a\mathbf{j}c} + J_{\mathbf{i}}^a J_{\mathbf{j}}^b B_{ab\mathbf{k}}}, \\
B^{3,0+0,3}_{\mathbf{ijk}} &\equiv  \tfrac{1}{4} \prs{B_{\mathbf{ijk}} - J_{\mathbf{i}}^a J_{\mathbf{j}}^b B_{ab\mathbf{k}} - J_{\mathbf{j}}^b J_{\mathbf{k}}^c B_{\mathbf{i}bc} -  J_{\mathbf{i}}^a J_{\mathbf{k}}^c B_{a \mathbf{j} c} }.
\end{split}
\end{align}
\begin{rmk}
Note as stated in (\cite{Gauduchon}, pp.263) , in $\dim M = 4$ we have $\Lambda_3^{3,0+0,3}$ is trivial.
\end{rmk}
Henceforth, we adhere to conventions of Gauduchon and Kobayashi--Nomizu (\cite{KN} Chapter IX, \cite{Gauduchon} pp.259 above (1.1.3)) regarding the corresponding relational identities. We define the $d^c$ operator acting on $\eta \in \Lambda_m(M)$ as
\begin{align*}
\prs{d^c \eta}_{\mathbf{a}_1 \cdots \mathbf{a}_{(n+1)}} &\triangleq - J_{\mathbf{a}_1}^{b_1} \cdots J_{\mathbf{a}_{(n+1)}}^{b_{(n+1)}} \prs{d \eta}_{\mathbf{b}_1 \cdots \mathbf{b}_{(n+1)}}.
\end{align*}
 The Nijenhuis tensor\footnote{Conventions on the Nijenhuis tensor $N$ match \cite{KN}, which differs with \cite{Gauduchon} by a factor of $8$.} is given by the following,
\begin{equation}\label{eq:Nijen}
N_{\mathbf{jk}}^{\mathbf{i}} \triangleq 2 \prs{ J_{\mathbf{j}}^p \prs{\del_p J_{\mathbf{k}}^{\mathbf{i}}} - J_{\mathbf{k}}^p \prs{\del_p J_{\mathbf{j}}^{\mathbf{i}}} - J_p^{\mathbf{i}} \prs{\del_{\mathbf{j}} J_{\mathbf{k}}^p} + J_p^{\mathbf{i}} \prs{\del_{\mathbf{k}} J_{\mathbf{j}}^p}}.
\end{equation}
The Nijenjuis tensor indicates the nonintegrability of $J$ and thus can be seen as one measure of `non-K\"{a}hlerness' of a manifold. Set $N_{\mathbf{ijk}} \triangleq N_{\mathbf{ij}}^l g_{l\mathbf{k}}$ and note $N$ (once its index is lowered) is $J$-antiinvariant in all pairs.
\[
N_{\mathbf{abc}} = - J_{\mathbf{a}}^u J_{\mathbf{b}}^v N_{uv\mathbf{c}} = - J_{\mathbf{b}}^v J_{\mathbf{c}}^r N_{\mathbf{a}vr}.
\]
The Nijenhuis tensor nearly satisfies a Bianchi identity up to a correction ((2.2.4) of \cite{Gauduchon}),
\begin{equation}\label{eq:Ncyclic}
N_{\mathbf{ijk}} + N_{\mathbf{jki}} + N_{\mathbf{kij}} = 8 \prs{d^c \gw}^{3,0+0,3}_{\mathbf{ijk}}.
\end{equation}
\begin{defn} The Chern connection is the unique connection $\N = \del + \gY$ such that
\begin{align}\label{eq:Cherncharac}
\begin{split}
\N g \equiv 0, \qquad \N J &\equiv 0,\qquad \tau^{1,1} \equiv 0,
\end{split}
\end{align}
where $\tau \in \Lambda_2 \otimes TM$ denotes the torsion of $\N$.
\end{defn}
\begin{rmk}
The Levi-Civita connection will be denoted by $D = \del + \gG$.
\end{rmk}
We begin by analysing key quantities of almost Hermitian manifolds with respect to the conventions of \cite{Gauduchon}. Set $\psi \equiv d \gw$.
\begin{defn} The \emph{Lee form} $\vartheta$ is given by
\begin{align*}
\vartheta_{\mathbf{k}} &
\triangleq \tfrac{1}{2} \gw^{ji} \psi_{ij\mathbf{k}}.
\end{align*}
\end{defn}
\subsection{Computational relations}
Given the objects established in the prior section, we now formulate key computational quantities and related identities which will be essential.
\begin{lemma}\label{lem:ThetaNPsi}For $(M^n, \gw, J)$ almost Hermitian, the negative contorsion tensor of the Chern connection is
\begin{align*}
\Theta_{ijk} &\triangleq \prs{D- \N}_{ijk} = \tfrac{1}{8}N_{jki} +  \tfrac{1}{2}  J_i^p \Psi_{pjk},
\end{align*}
where
\begin{equation}\label{eq:Psidef}
\Psi_{ijk} \triangleq \tfrac{1}{2} \prs{\psi_{ijk} + J_j^q J_k^r\psi_{iqr} + J_i^q J_k^r \psi_{qjr} +J_i^p J_j^q \psi_{pqk} } \in \Lambda_3.
\end{equation}
\begin{proof}
Starting from formula (2.5.3) of \cite{Gauduchon} which presents a family of Hermitian connections parametrized by the variable $t$, taking $t = 1$ for the Chern connection we obtain
\begin{align}
\begin{split}\label{eq:gYgG}
\prs{ \N - D}_{ijk} &= - \tfrac{1}{8} N_{jki} + \tfrac{1}{2} \prs{d^c \gw}^{3,0+0,3}_{ijk} + \tfrac{1}{2} \prs{d^c \gw}^{2,1+1,2}_{ibc}J_j^b J_k^c.
\end{split}
\end{align}
Thus rearranging yields
\begin{align*}
\Theta_{ijk} &= \tfrac{1}{8}N_{jki} -  \tfrac{1}{2} \prs{\prs{d^c \gw}^{3,0+0,3}_{ijk} + \prs{d^c \gw}^{2,1+1,2}_{ibc}J_j^b J_k^c}.
\end{align*}
We expand out the latter terms on the right hand side. First,
\begin{align}
\begin{split}\label{eq:dc1221}
\prs{d^c \gw}^{2,1+1,2}_{ijk} &= \tfrac{3}{4}\prs{d^c \gw}_{ijk} + \tfrac{1}{4} \prs{J_j^b J_k^c \prs{d^c \gw}_{ibc} + J_i^a J_k^c \prs{d^c \gw}_{ajc} + J_i^a J_j^b \prs{d^c \gw}_{abk}}\\
&= -\tfrac{3}{4}J_i^p J_j^q J_k^r\psi_{pqr} - \tfrac{1}{4} \prs{  J_i^p  \psi_{pjk} +  J_j^q  \psi_{iqk} +   J_k^r \psi_{ijr}}.
\end{split}
\end{align}
Next we compute
\begin{align}
\begin{split}\label{eq:dcw3003}
\prs{d^c \gw}^{3,0+0,3}_{ijk} &= \tfrac{1}{4} \prs{\prs{d^c \gw}_{ijk} - J_i^a J_j^b \prs{d^c \gw}_{abk} - J_j^b J_k^c \prs{d^c \gw}_{ibc} -  J_i^a J_k^c \prs{d^c \gw}_{ajc} } \\
&= \tfrac{1}{4} \prs{-J_i^p J_j^q J_k^r\psi_{pqr} +   J_k^r \psi_{ijr} +   J_i^p \psi_{pjk} +  J_j^q \psi_{iqk} }.
\end{split}
\end{align}
Modifying \eqref{eq:dc1221} with action by the almost complex structure yields:
\begin{align}
\begin{split}\label{eq:JJdcw1221}
J_j^m J_k^n\prs{d^c \gw}^{2,1+1,2}_{imn}
&= -\tfrac{3}{4}J_j^m J_k^n J_i^p J_m^q J_n^r\psi_{pqr} - \tfrac{1}{4}J_j^m J_k^n \prs{  J_i^p  \psi_{pmn} +  J_m^q  \psi_{iqn} +   J_n^r \psi_{imr}} \\
&= -\tfrac{3}{4} J_i^p \psi_{pjk} - \tfrac{1}{4} \prs{ J_j^m J_k^n J_i^p  \psi_{pmn} -  J_k^n  \psi_{ijn} - J_j^m \psi_{imk}}.
\end{split}
\end{align}
Combining \eqref{eq:JJdcw1221} and \eqref{eq:dcw3003} yields
\begin{align*}
J_j^m J_k^n\prs{d^c \gw}^{2,1+1,2}_{imn}  + \prs{d^c \gw}^{3,0+0,3}_{ijk} &=- \tfrac{1}{4}  J_i^p J_j^q J_k^r\psi_{pqr} + \tfrac{1}{4}  J_k^r \psi_{ijr} + \tfrac{1}{4} J_i^p \psi_{pjk} + \tfrac{1}{4}   J_j^q \psi_{iqk}\\
&\hsp -\tfrac{3}{4} J_i^p \psi_{pjk}  - \tfrac{1}{4} J_j^m J_k^n J_i^p  \psi_{pmn} + \tfrac{1}{4}J_k^n  \psi_{ijn} + \tfrac{1}{4}J_j^m \psi_{imk} \\
&= - \tfrac{1}{2} J_i^p \prs{ J_j^q J_k^r\psi_{pqr} + J_p^y J_k^r \psi_{yjr} +J_p^y J_j^q \psi_{yqk} + \psi_{pjk}}.
\end{align*}
Inserting this into \eqref{eq:gYgG} yields the result.
\end{proof}
\end{lemma}
\begin{cor}\label{cor:Ncycleid}
Updating \eqref{eq:Ncyclic} with \eqref{eq:dcw3003} yields that
\begin{align*}
N_{ijk} + N_{jki} + N_{kij} =  2 \prs{J_k^r \psi_{ijr} +   J_i^p \psi_{pjk} +   J_j^q \psi_{iqk} -J_i^p J_j^q J_k^r\psi_{pqr}}.
\end{align*}
\end{cor}
With the tensors $\Theta$ and $\Psi$ classified above, we record some basic properties of each which will be crucial to upcoming computations. Note that $\Psi$ was identified as a natural generalisation of $\psi$; the below properties for $\Psi$ hold for $\psi$, and for $\dim M = 4$, we have $\Psi \equiv \psi$.
\begin{lemma}\label{lem:ThetaPsi} For $(M^n,\gw,J)$ almost Hermitian,
\begin{enumerate}
\item $\Theta_{ijk}$ is skew in $(jk)$.
\item $\gw^{ij} \Psi_{ijk} = -2 \vartheta_k$ and $g^{ij} \Psi_{ijk} = 0$.
\item $\gw^{ij} \Theta_{ijk} = 0$, $\gw^{ik} \Theta_{ijk} = 0$, and $\gw^{jk} \Theta_{ijk} = -J_i^p \vartheta_p$.
\item $g^{ij} \Theta_{ijk} = - \vartheta_k$, $g^{ik} \Theta_{ijk} = \vartheta_j$, and $g^{jk} \Theta_{ijk} = 0$.
\end{enumerate}
\begin{proof} Identity (1) is apparent by the definition of $\Theta$. For (2), we compute the corresponding contractions of $\Psi$, beginning with contractions by $\gw$,
\begin{align*}
\gw^{ij} \Psi_{ijk} &= \tfrac{1}{2} \prs{ J_j^q \gw^{ij} J_k^r\psi_{iqr} + J_i^p\gw^{ij} J_k^r \psi_{pjr} +J_i^p\gw^{ij} J_j^q \psi_{pqk} + \gw^{ij}\psi_{ijk}} \\
&= \tfrac{1}{2} \prs{ g^{qi} J_k^r\psi_{iqr} - g^{pj} J_k^r \psi_{pjr} + \gw^{pq} \psi_{pqk} + \gw^{ij}\psi_{ijk}} \\
&= -2 \vartheta_k.
\end{align*}
The $g$ contraction follows easily from the expression of $\Psi$ in \eqref{eq:Psidef}.
For (3) we investigate $\gw$ contractions of $\Theta$. First,
\begin{align*}
\gw^{ij}\Theta_{ijk} &= 0 +  \tfrac{1}{2}  J_i^p \gw^{ij} \Psi_{pjk} \\
&= 0.
\end{align*}
The next identity, $\gw^{ik} \Theta_{ijk} = 0$ follows by (1). Lastly we have
\begin{align*}
\gw^{jk}\Theta_{ijk} &= 0 + \tfrac{1}{2}  J_i^p \gw^{jk} \Psi_{pjk} \\
&= - J_i^p \vartheta_p.
\end{align*}
Finally, we have that for (4),
\begin{align*}
g^{ij}\Theta_{ijk} &= 0 +  \tfrac{1}{2} J_i^p g^{ij} \Psi_{pjk} \\
&= - \vartheta_k.
\end{align*}
The last two identities follow by (1) again. We conclude the result.
\end{proof}
\end{lemma}
\subsection{Chern curvature identities}
Here we record some main identities concerning the Chern connection's associated curvature $\Omega$, torsion $\tau$ and contorsion tensor $-\Theta$.
\begin{lemma}\label{lem:Cherncurvsym}
Let $\prs{M^{2n},g,J}$ be almost Hermitian. Then
\begin{align*}
\Omega_{\mathbf{ijkl}} = - \Omega_{\mathbf{jikl}} = - \Omega_{\mathbf{ijlk}} = \Omega_{\mathbf{ij}ab} J^a_{\mathbf{k}} J^b_{\mathbf{l}} .
\end{align*}
\end{lemma}
\noindent Aspects of the Riemann curvature tensor translate with some residual torsion terms. Define \emph{Chern--Ricci curvature} ($\cP$) and a computational intermediary ($\cV$) by
\begin{align}
\begin{split}\label{eq:curvatures}
\cP_{\mathbf{ab}} &\triangleq  \gw^{cd}  \Omega_{\mathbf{ab}cd} \qquad
\cV_{\mathbf{ab}}\triangleq  \Omega_{r \mathbf{ab}}^r,
\end{split}
\end{align}
and finally, the \emph{Chern scalar curvature} is
\begin{equation*}
\varrho \triangleq  \gw^{ba} \cP_{ab}.
\end{equation*}
\begin{prop}\label{prop:RmOmega} For $(M^n, \gw, J)$ almost Hermitian,
\begin{equation}\label{eq:RmvsOmega}
\Rm_{ijkl} = \Omega_{ijkl} + \N_i \Theta_{jkl} - \N_j \Theta_{ikl} + \Theta_{isl} \Theta_{jk}^s - \Theta_{jsl} \Theta_{ik}^s + \tau_{ijs} \Theta_{skl}.
\end{equation}
\begin{proof} This follows by expanding out the formulas for $\Rm$ and $\Omega$ in terms of connection coefficients.
\end{proof}
\end{prop}
\begin{cor}\label{cor:RcncV}
For $(M^n, \gw, J)$ almost Hermitian,
\begin{align*}
\Rc_{jk}
&= \cV_{jk} + g^{il}\N_i \Theta_{jkl} - \N_j \vartheta_k + \vartheta_s\Theta_{jk}^s - \Theta_{ij}^s\Theta_{sk}^i.
\end{align*}
\begin{proof}
Take the trace of Proposition \ref{prop:RmOmega} and simplify using Lemma \ref{lem:ThetaPsi}.
\end{proof}
\end{cor}
\begin{lemma}\label{lem:omegacyc} For $(M^n, J, \gw)$ almost Hermitian,
\begin{align*}
\Omega_{ijkl} + \Omega_{jkil} + \Omega_{kijl} &= \prs{\N_i \tau_{jkl}+  \N_j \tau_{kil} + \N_k \tau_{ijl}} + \prs{\tau_{jk}^s\tau_{sil} + \tau_{ki}^s \tau_{sjl}+ \tau_{ij}^s \tau_{skl}}.
\end{align*}
\begin{proof} Expand the Bianchi identity for Riemannian curvature using Proposition \ref{prop:RmOmega}.
\end{proof}
\end{lemma}
\begin{lemma}\label{lem:Chernsymm}
For $(M^n, J, \gw)$ almost Hermitian we have that
\begin{align*}
\Omega_{klij} -\Omega_{ijkl} \triangleq T_{klij} &= \N_i \Theta_{jkl} - \N_j \Theta_{ikl} -\N_k \Theta_{lij} + \N_l \Theta_{kij} \\
&\hsp + \Theta_{isl} \Theta_{jk}^s - \Theta_{ksj} \Theta_{li}^s - \Theta_{jsl} \Theta_{ik}^s+ \Theta_{lsj} \Theta_{ki}^s + \tau_{ij}^s \Theta_{skl}  - \tau_{kl}^s \Theta_{sij}.
\end{align*}
Where it follows that
\begin{equation}\label{eq:Tsyms}T_{ijkl} = - T_{klij}, \qquad T_{ijkl} = - T_{jikl} =- T_{ijlk}.
\end{equation}
\begin{proof}
This follows by computing $\Rm_{ijkl} - \Rm_{klij}$ using \eqref{eq:RmvsOmega}.
\end{proof}
\end{lemma}
\section{Relation of Ricci tensors} To show that \eqref{eq:AHCF} preserves the almost Hermitian structure, we need to obtain the explicit form of the difference between the Ricci and Chern Ricci tensors.
\begin{lemma}\label{lem:VnJP} For $(M^n, J,\gw)$ almost Hermitian,
\begin{align*}
\cV_{jk}
 &= \tfrac{1}{2}J_k^a \cP_{ja} + J_k^a\gw^{br} \N_b \Theta_{jar}  + \tfrac{1}{2}\N_j   \vartheta_k+ \tfrac{1}{2}J_k^a  J_j^u\N_a \vartheta_u \\
 &\hsp +  J_k^a \gw^{br} \prs{\Theta_{ar}^s - \Theta_{ra}^s}  \Theta_{jbs} +J_k^a\gw^{br} \prs{\Theta_{jb}^s - \Theta_{bj}^s} \Theta_{sar} + \tfrac{1}{2} J_k^a J_s^p\vartheta_p \prs{\Theta_{ja}^s - \Theta_{aj}^s}.
\end{align*}
\begin{proof} We model our argument on Lemma 2.16 of \cite{Kelleher}. Using Lemma \ref{lem:omegacyc} and manipulating (with an important step highlighted),
\begin{align*}
\cV_{jk} &= g^{re} \Omega_{rjke} \\
&= J_k^a J_e^b g^{re} \Omega_{rjab} \\
&= \underbrace{J_k^a \gw^{br} \Omega_{rjab}} \\
&= - J_k^a \gw^{br} \prs{\Omega_{jarb} + \Omega_{arjb} } \\
&\hsp + J_k^a \gw^{br} \prs{\N_j \tau_{arb} + \N_a \tau_{rjb} + \N_r \tau_{jab}} \\
&\hsp - J_k^a \gw^{br} g_{bs} \prs{ \tau_{re}^s \tau_{ja}^e + \tau_{je}^s \tau_{ar}^e + \tau_{ae}^s \tau_{rj}^e}.
\end{align*}

We leave the latter lines alone and manipulate the first,
\begin{align*}
 - J_k^a \gw^{br} \prs{ \Omega_{jarb} + \Omega_{arjb}}&=  J_k^a \cP_{ja} - J_k^a \gw^{br}  \Omega_{arjb}  \\
&=  J_k^a \cP_{ja} - J_k^a \gw^{br}  T_{arjb} - J_k^a \gw^{br}  \Omega_{jbar}  \\
&=  J_k^a \cP_{ja} - J_k^a \gw^{br}  T_{arjb} - \underbrace{J_k^a \gw^{br}  \Omega_{rjab}}.
\end{align*}
Given the reappearance of the indicated term, we rearrange and obtain
\begin{align*}
\cV_{jk} &= J_k^a \gw^{br} \Omega_{rjab}\\
&=  \tfrac{1}{2}J_k^a \cP_{ja} - \tfrac{1}{2} J_k^a \gw^{br}  T_{arjb} \\
&\hsp +\tfrac{1}{2} J_k^a \gw^{br} \prs{\N_j \tau_{arb} + \N_a \tau_{rjb} + \N_r \tau_{jab}} \\
&\hsp - \tfrac{1}{2} J_k^a \gw^{br} g_{bs} \prs{ \tau_{re}^s \tau_{ja}^e + \tau_{je}^s \tau_{ar}^e + \tau_{ae}^s \tau_{rj}^e}.
\end{align*}
Now we convert all $\tau$ terms to $\Theta$, noting that $\tau_{ijk} \equiv \Theta_{jik} - \Theta_{ijk}$,
\begin{align*}
\cV_{jk}
&=  \tfrac{1}{2}J_k^a \cP_{ja} - \tfrac{1}{2} J_k^a \gw^{br}  T_{arjb} &\mbox{(R1)}\\
&\hsp +\tfrac{1}{2} J_k^a \gw^{br} \prs{-\N_j  \Theta_{arb} + \N_a \Theta_{jrb} + \N_r \prs{\Theta_{ajb}- \Theta_{jab}}} &\mbox{(R2)}\\
&\hsp - \tfrac{1}{2} J_k^a \gw^{br}\prs{ \prs{\Theta_{erb}- \Theta_{reb} }\prs{\Theta_{aj}^e -\Theta_{ja}^e} + \prs{\Theta_{ejb} -\Theta_{jeb}} \prs{ \Theta_{ra}^e - \Theta_{ar}^e}} \\
&\hsp \hsp - \tfrac{1}{2} J_k^a \gw^{br}\prs{\prs{\Theta_{eab} - \Theta_{aeb}} \prs{\Theta_{jr}^e -\Theta_{rj}^e}}. &\mbox{(R3)}
\end{align*}
We now simplify each labelled row beginning with Row 1. Using Lemmata \ref{lem:ThetaPsi} and \ref{lem:Chernsymm},
\begin{align*}
\gw^{br}T_{arjb}  &= -\gw^{br}T_{arbj}\\
&=- \gw^{br} \N_b \Theta_{jar} + \gw^{br} \N_j \Theta_{bar} + \gw^{br} \N_a \Theta_{rbj} - \gw^{br} \N_r \Theta_{abj} \\
&\hsp - \gw^{br} \Theta_{bsr} \Theta_{ja}^s + \gw^{br} \Theta_{asj} \Theta_{rb}^s + \gw^{br} \Theta_{jsr} \Theta_{ba}^s- \gw^{br}\Theta_{rsj} \Theta_{ab}^s \\
&\hsp - \gw^{br}\tau_{bj}^s \Theta_{sar}  + \gw^{br}\tau_{ar}^s \Theta_{sbj} \\
&= - \gw^{br} \N_b \Theta_{jar} + \gw^{br} \N_b \Theta_{arj} - \gw^{br} \Theta_{jsb} \Theta_{ra}^s + \gw^{br}\Theta_{bsj}\Theta_{ar}^s\\
&\hsp - \gw^{br} \Theta_{jb}^s \Theta_{sar} + \gw^{br} \Theta_{bj}^s \Theta_{sar} + \gw^{br} \Theta_{sbj} \Theta_{ra}^s - \gw^{br} \Theta_{sbj} \Theta_{ar}^s.
\end{align*}
Therefore it follows that
\begin{align*}
\prs{\text{R1}}_{jk} -  \tfrac{1}{2}J_k^a \cP_{ja}&=- \tfrac{1}{2} J_k^a\gw^{br}T_{arjb} \\
&= \tfrac{1}{2} J_k^a\gw^{br} \N_b \Theta_{jar} - \tfrac{1}{2} J_k^a\gw^{br} \N_b \Theta_{arj}
 + \tfrac{1}{2} J_k^a\gw^{br} \Theta_{jsb} \Theta_{ra}^s - \tfrac{1}{2} J_k^a \gw^{br}\Theta_{bsj}\Theta_{ar}^s\\
 &\hsp +\tfrac{1}{2} J_k^a\gw^{br} \Theta_{jb}^s \Theta_{sar} - \tfrac{1}{2} J_k^a \gw^{br} \Theta_{bj}^s \Theta_{sar}
 - \tfrac{1}{2} J_k^a \gw^{br} \Theta_{sbj} \Theta_{ra}^s + \tfrac{1}{2} J_k^a\gw^{br} \Theta_{sbj} \Theta_{ar}^s \\
 &= \tfrac{1}{2} J_k^a\gw^{br} \N_b \Theta_{jar} - \tfrac{1}{2} J_k^a\gw^{br} \N_b \Theta_{arj}
 + \tfrac{1}{2} J_k^a\gw^{br} \Theta_{jsb} \Theta_{ra}^s - \tfrac{1}{2} J_k^a \gw^{br}\Theta_{bsj}\Theta_{ar}^s\\
 &\hsp +\tfrac{1}{2} J_k^a\gw^{br} \prs{\Theta_{jb}^s - \Theta_{bj}^s} \Theta_{sar}
 + \tfrac{1}{2} J_k^a \gw^{br} \Theta_{sbj}\prs{\Theta_{ar}^s -\Theta_{ra}^s}.
\end{align*}
We now address Row 2. For this, applying Lemma \ref{lem:ThetaPsi}.
\begin{align*}
\prs{\text{R2}}_{jk} &= \tfrac{1}{2} J_k^a \gw^{br} \prs{-\N_j  \Theta_{arb} + \N_a \Theta_{jrb} + \N_r \prs{\Theta_{ajb}- \Theta_{jab}}}\\
&=\tfrac{1}{2}\N_j   \vartheta_k+ \tfrac{1}{2}J_k^a  J_j^u\N_a \vartheta_u - \tfrac{1}{2}J_k^a \gw^{br} \N_b \Theta_{ajr}+ \tfrac{1}{2}J_k^a \gw^{br} \N_b\Theta_{jar}.
\end{align*}
\noindent Lastly, for Row 3, we expand out and search for simplifications
\begin{align*}
\prs{\text{R3}}_{jk} &= -\tfrac{1}{2} J_k^a \gw^{br} \prs{ \prs{\Theta_{srb}- \Theta_{rsb} }\prs{\Theta_{aj}^s -\Theta_{ja}^s} + \prs{\Theta_{sjb} -\Theta_{jsb}} \prs{ \Theta_{ra}^s - \Theta_{ar}^s}} \\
&\hsp  + \tfrac{1}{2} J_k^a \gw^{br}\prs{\prs{\Theta_{sar} - \Theta_{asr}} \prs{\Theta_{jb}^s -\Theta_{bj}^s}} \\
 &=  -\tfrac{1}{2} J_k^a J_s^p\vartheta_p \prs{\Theta_{aj}^s -\Theta_{ja}^s} - \tfrac{1}{2} J_k^a \gw^{br} \Theta_{sbj} \prs{ \Theta_{ar}^s -  \Theta_{ra}^s} -\tfrac{1}{2} J_k^a \gw^{br}  \Theta_{jsb} \prs{ \Theta_{ar}^s -  \Theta_{ra}^s} \\
&\hsp  + \tfrac{1}{2} J_k^a \gw^{br}\Theta_{sar}\prs{\Theta_{jb}^s -\Theta_{bj}^s}  -  \tfrac{1}{2} J_k^a \gw^{br}\Theta_{asr} \prs{\Theta_{jb}^s -\Theta_{bj}^s}.
\end{align*}
Combining everything we have that
\begin{align*}
\cV_{jk}
 &= \tfrac{1}{2}J_k^a \cP_{ja} + J_k^a\gw^{br} \N_b \Theta_{jar}  + \tfrac{1}{2}\N_j   \vartheta_k+ \tfrac{1}{2}J_k^a  J_j^u\N_a \vartheta_u \\
 &\hsp +  J_k^a \gw^{br} \prs{\Theta_{ar}^s - \Theta_{ra}^s}  \Theta_{jbs} +J_k^a\gw^{br} \prs{\Theta_{jb}^s - \Theta_{bj}^s} \Theta_{sar} + \tfrac{1}{2} J_k^a J_s^p\vartheta_p \prs{\Theta_{ja}^s - \Theta_{aj}^s},
\end{align*}
as desired.
\end{proof}
\end{lemma}
\begin{cor}\label{cor:RcJPtotdif} For $(M^n, \gw, J)$ almost Hermitian,
\begin{align*}
\Rc_{jk}
&= \tfrac{1}{2}J_k^a \cP_{ja}  - \tfrac{1}{2}\N_j   \vartheta_k+ \tfrac{1}{2}J_k^a  J_j^u\N_a \vartheta_u + g^{il}\N_i \Theta_{jkl} + J_k^a\gw^{br} \N_b \Theta_{jar} \\
 &\hsp +  J_k^a \gw^{br} \prs{\Theta_{ar}^s - \Theta_{ra}^s}  \Theta_{jbs} +J_k^a\gw^{br} \prs{\Theta_{jb}^s - \Theta_{bj}^s} \Theta_{sar} + \tfrac{1}{2} J_k^a J_s^p\vartheta_p \prs{\Theta_{ja}^s - \Theta_{aj}^s} \\
 &\hsp + \vartheta_s\Theta_{jk}^s - \Theta_{ij}^s\Theta_{sk}^i.
\end{align*}
\begin{proof}
Combining Lemma \ref{lem:VnJP} with Corollary \ref{cor:RcncV} yields the result.
\end{proof}
\end{cor}
Provided the explicit comparison formula of the Ricci tensors for both connections, we now focus in on analysing the highest order terms in preparation for future symbol computations. We let $\brk{ \cdot }_2$ denotes the projection to second derivatives of the primitives $(g,J)$.
\begin{prop}\label{prop:RcJPasN} For $(M^n, \gw, J)$ almost Hermitian,
\begin{align}
\begin{split}\label{prop:Rcdiffagain}
\brk{\Rc_{jk} - \tfrac{1}{2}J_k^a \cP_{ja} }_{2} &=  \brk{\tfrac{1}{4}   \N^l N_{klj}  -\tfrac{1}{2} \N_j   \vartheta_k + \tfrac{1}{2} J_k^a  J_j^y \N_a \vartheta_y - \tfrac{1}{2} \N^e \prs{ J_e^r \psi_{jkr} + J_k^a \psi_{jae}}}_2.
\end{split}
\end{align}
\begin{proof} Isolating the highest order terms of Corollary \ref{cor:RcJPtotdif} yields
\begin{align}
\begin{split}\label{eq:RcVHO}
\brk{\Rc_{jk}}_{2}
&= \brk{\tfrac{1}{2}J_k^a \cP_{ja} - \tfrac{1}{2}\N_j   \vartheta_k+ \tfrac{1}{2}J_k^a  J_j^u\N_a \vartheta_u + g^{il}\N_i \Theta_{jkl} + J_k^a\gw^{br} \N_b \Theta_{jar}}_2.
\end{split}
\end{align}
We will expand out the last two terms involving $\Theta$. Using Lemma \ref{lem:ThetaNPsi} we have
\begin{align*}
g^{il}\N_i \Theta_{jkl} &= \tfrac{1}{8}\N^l  N_{klj} + \tfrac{1}{2} J_j^u \N^l \Psi_{ukl}.
\end{align*}
Next we compute
\begin{align*}
J_k^a\gw^{br} \N_b \Theta_{jar} &= J_k^a\gw^{br} \N_b \prs{\tfrac{1}{8} N_{arj} + \tfrac{1}{2} J_j^u \Psi_{uar}} \\
&= J_k^a\gw^{br} \N_b \prs{ - \tfrac{1}{8} J_a^m J_r^n N_{mnj} + \tfrac{1}{2} J_j^u \Psi_{uar}} \\
&=   \tfrac{1}{8}   \N^l N_{klj} + \tfrac{1}{2}  J_k^a\gw^{br} J_j^u \N_b \Psi_{uar}.
\end{align*}
Thus, updating \eqref{eq:RcVHO} we conclude that
\begin{align*}
\brk{\Rc_{jk} - \tfrac{1}{2}J_k^a \cP_{ja} }_{2} = \brk{  \tfrac{1}{4}   \N^l N_{klj}  -\tfrac{1}{2} \N_j   \vartheta_k + \tfrac{1}{2} J_k^a  J_j^y \N_a \vartheta_y+ \tfrac{1}{2} J_j^u \prs{\N^l \Psi_{ukl} - \tfrac{1}{2}  J_k^a J^r_e  \N^e \Psi_{uar}}}_2.
\end{align*}
We will address the last quantity. For this, first consider
\begin{align*}
\N^e \Psi_{uke} &= \tfrac{1}{2}\N^e \prs{J_k^a J_e^r\psi_{uar} + J_u^y J_e^r \psi_{ykr} +J_u^y J_k^a \psi_{yae} + \psi_{uke}}.
\end{align*}
Next we compute out,
\begin{align*}
 J_k^a J^r_e \N^e \Psi_{uar} &= \tfrac{1}{2} \N^e \prs{ J_k^a J^r_e J_a^q J_r^h\psi_{uqh} +  J_k^a J^r_e J_u^y J_r^h \psi_{yah} + J_k^a J^r_e J_u^y J_a^q \psi_{yqr} +  J_k^a J^r_e \psi_{uar}} \\
 &= \tfrac{1}{2}\N^e \prs{  \psi_{uke} -  J_k^a J_u^y \psi_{yae} - J^r_e J_u^y  \psi_{ykr} +  J_k^a J^r_e \psi_{uar}}.
\end{align*}
Combining we obtain that
\begin{align*}
\N^e \Psi_{uke} -  J_k^a J^r_e \N^e \Psi_{uar} &= \N^e \prs{ J_u^y J_e^r \psi_{ykr} +J_u^y J_k^a \psi_{yae}}.
\end{align*}
Combining yields the result.
\end{proof}
\end{prop}
\section{Construction of the flow}
We aim to derive the structure of a flow for almost complex structures which matches with Ricci flow on the metrics (as displayed in \eqref{eq:AHCF}). To do so, we begin investigating the operator introduced in the work of of the symplectic curvature flow \cite{ST} (pp. 182) but rather in the almost Hermitian setting and modify accordingly. Let us write the flow we will construct as follows:
\begin{equation}\label{eq:AHCF-1}
\begin{cases}
\prs{\tfrac{\del g}{\del t}}_{bc} &= - 2 \Rc_{bc} \\
\prs{\tfrac{\del J}{\del t}}_{a}^c &= - \prs{\cP_{av} - 2 J_a^y \Rc_{yv} }^{2,0+0,2} g^{vc} + \mho_{av} g^{vc},
\end{cases}
\end{equation}
where $\mho \in TM_{\otimes 2}$ will be determined. The first term in the flow on $J$ corresponds to the operator
\begin{align}
\begin{split}\label{eq:SCFoperater}
J_a^c \mapsto J_a^k \prs{2 \Rc_{jk}- J_k^a \cP_{ja}}g^{jc}.
\end{split}
\end{align}
In order to preserve the properties of $J_t$ as an almost complex structure, the right side of the flow on $J$ in \eqref{eq:AHCF-1} must be $J$-antiinvariant after we lower the superscript to subscript by the metric $g$. On other other hand, if we require \eqref{eq:AHCF-1} to preserve the Hermitian structure, we have
\begin{align*}
0 = \tfrac{\del}{\del t}\brk{g_{ab} - g_{pq} J_a^p J_b^q}
&=   -4\prs{\Rc_{ab}}^{2,0+0,2}  - g_{pq} g^{vp} \prs{ -\cP_{av}^{2,0+0,2} + 2 J_a^e \Rc^{2,0+0,2}_{ev} + \mho_{av}}J_b^q\\
&\hsp -  g_{pq} g^{qv}J_a^p \prs{-\cP_{bv}^{2,0+0,2} + 2 J_b^e \Rc_{ev}^{2,0+0,2} + \mho_{bv}} \\
&=  - 4\Rc_{ab}^{2,0+0,2} - g_{pq} g^{vp} \prs{ -\cP_{av}^{2,0+0,2} + 2 J_a^e \Rc^{2,0+0,2}_{ev} + \mho_{av}}J_b^q\\
&\hsp -  g_{pq} g^{qv}J_a^p \prs{-\cP_{bv}^{2,0+0,2} + 2 J_b^e \Rc_{ev}^{2,0+0,2} + \mho_{bv}} \\
&=  - 4\Rc_{ab}^{2,0+0,2} - \prs{ -\cP_{aq}^{2,0+0,2} + 2 J_a^e \Rc^{2,0+0,2}_{eq} + \mho_{aq}}J_b^q\\
&\hsp -  J_a^p \prs{-\cP_{bp}^{2,0+0,2} + 2 J_b^e \Rc_{ep}^{2,0+0,2} + \mho_{bp}} \\
&= - J_b^q\prs{ \mho_{aq} + \mho_{qa}}.
\end{align*}
Therefore
\begin{align*}
\prs{\mho_{aw}}_{\text{sym}} &= 0.
\end{align*}
This means that $\mho$ must be skew-symmetric, but we have freedom in choosing $\mho$ so long it is skew-symmetric and $J$-antiinvariant. As we tailor $\mho$ to produce the proper (psuedo)parabolic flow, we need to compute the symbol of the operator in \eqref{eq:SCFoperater}. Using Proposition \ref{prop:RcJPasN} above we reexpress this, up to highest order, as
\begin{align}
\begin{split}\label{eq:operator decomp}
\brk{ 2 J_a^k \prs{ \Rc_{jk}- \tfrac{1}{2} J_k^v \cP_{jv} }^{2,0+0,2} g^{jc}}_2
&=  \brk{-\tfrac{1}{2}  J_a^k \N^l N_{lk}^c+ \tfrac{1}{2}\prs{\mathcal{A}(J)}_a^c + \tfrac{1}{2} \prs{\mathcal{B}(J)}_a^c}_2,
\end{split}
\end{align}
where we have set
\begin{align}
\begin{split}\label{eq:Psi}
\tfrac{1}{2}\prs{\mathcal{A}(J)}_a^c &\triangleq
 2 \gw^{cy} \prs{\N_y   \vartheta_a}^{2,0+0,2}_{\text{sym}},  \\
\tfrac{1}{2} \prs{\mathcal{B}(J)}_a^c &\triangleq g^{jc} \prs{\N^e \psi_{jae} -  J_a^kJ_e^r\N^e \psi_{jkr}}^{2,0+0,2}.
\end{split}
\end{align}
Note that both $\mathcal{A}$ and $\mathcal{B}$ vanish in the almost K\"{a}hler setting. Since we know the behaviour of the first term via Proposition 5.4 of \cite{ST}, we will investigate these latter terms. We begin with an identification of $\mathcal{A}$ which will be relevant to us later.

\begin{lemma}\label{lem:A(J)class}
We have that
\begin{equation*}
\tfrac{1}{2}\prs{\mathcal{A}(J)}_a^c = \gw^{cy} \prs{\mathcal{L}_{\vartheta} g}^{2,0+0,2}_{ya} - \tfrac{1}{8} \gw^{cy} \vartheta^d \prs{N_{dya} + N_{day}}.
\end{equation*}
\begin{proof}
We first write the Lie derivative using the Levi-Civita connection and then convert to the Chern connection using Lemma \ref{lem:ThetaNPsi}.
\begin{align}\label{eq:Lvarthetag}
\begin{split}
\prs{\mathcal{L}_{\vartheta} g}_{ij} &= \prs{D_i \vartheta_j} + \prs{D_j \vartheta_i} \\
&= \prs{\N_i \vartheta_j}+ \prs{\N_j \vartheta_i} - \Theta_{ijs} \vartheta^s  - \Theta_{jis} \vartheta^s \\
&= \prs{\N_i \vartheta_j}+ \prs{\N_j \vartheta_i} - \prs{\tfrac{1}{8} N_{jdi} + \tfrac{1}{2} J_i^p \Psi_{pjd} }\vartheta^d  - \prs{\tfrac{1}{8} N_{idj} + \tfrac{1}{2} J_j^p \Psi_{pid} }\vartheta^d \\
&= \prs{ \prs{\N_i \vartheta_j} + \prs{\N_j \vartheta_i}} + \tfrac{1}{8} \vartheta^d \prs{N_{dji} + N_{dij}} + \vartheta^d \prs{J_i^p \Psi_{pdj} + J_j^p \Psi_{pdi}}.
\end{split}
\end{align}
%
%
Then, projecting onto the $(2,0)+(0,2)$ component we have that
\begin{align*}
\prs{\mathcal{L}_{\vartheta} g}^{2,0+0,2}_{ij} &= \prs{ \prs{\N_i \vartheta_j} + \prs{\N_j \vartheta_i}}^{2,0+0,2} + \tfrac{1}{8} \vartheta^d \prs{N_{dji} + N_{dij}} + \vartheta^d \prs{J_i^p \Psi_{pdj} + J_j^p \Psi_{pdi}}^{2,0+0,2}.
\end{align*}
Now, for the final term we note that in fact
\begin{align*}
2 \prs{J_i^p \Psi_{pdj} + J_j^p \Psi_{pdi}}^{2,0+0,2} &= J_i^p \Psi_{pdj} + J_j^p \Psi_{pdi} - J_i^m J_j^n J_m^p \Psi_{pdn} - J_i^m J_j^n J_n^p \Psi_{pdm} \\
&= J_i^p \Psi_{pdj} + J_j^p \Psi_{pdi} - J_i^m J_j^n J_m^p \Psi_{pdn} - J_i^m J_j^n J_n^p \Psi_{pdm} \\
&= J_i^p \Psi_{pdj} + J_j^p \Psi_{pdi} +  J_j^p \Psi_{idp} + J_i^p  \Psi_{jdp} \\
&= J_i^p \Psi_{pdj} + J_j^p \Psi_{pdi} -  J_j^p \Psi_{pdi} - J_i^p  \Psi_{pdj} \\
&= 0.
\end{align*}
So updating \eqref{eq:Lvarthetag} further yields
\begin{align*}
\prs{\mathcal{L}_{\vartheta} g}^{2,0+0,2}_{ij} &= 2 \prs{ \prs{\N_i \vartheta_j} }^{2,0+0,2}_{sym} + \tfrac{1}{8} \vartheta^d \prs{N_{dji} + N_{dij}}.
\end{align*}
In particular, it follows that
\begin{equation*}
\tfrac{1}{2}\prs{\mathcal{A}(J)}_a^c = 2 \gw^{cy} \prs{\N_y \vartheta_a}^{2,0+0,2}_{\text{sym}} = \gw^{cy} \prs{\mathcal{L}_{\vartheta} g}^{2,0+0,2}_{ya} - \tfrac{1}{8} \gw^{cy} \vartheta^d \prs{N_{dya} + N_{day}},
\end{equation*}
which yields the result.
\end{proof}
\end{lemma}
For the following computations we now focus on analysing the symbol of the operator $\mathcal{A}(J)$. Let $\brk{\cdot}_{2/g}$ denote the projection onto second order operators in $J$, \emph{excluding} second order operators on $g$.
\begin{prop}\label{prop:dim4leeform}
The symbol of the operator $J \mapsto \mathcal{A} (J)$ is trivial.
\end{prop}
\begin{proof}
For this we expand out
\begin{align}
\begin{split}\label{eq:Psidecomp}
(\mathcal{A} (J))_a^c =  4 \gw^{cb} \prs{\N_b   \vartheta_a}^{2,0+0,2}_{\text{sym}} =  2 \gw^{cb} \prs{\N_a   \vartheta_b + \N_b  \vartheta_a}^{2,0+0,2}
= \prs{\mathcal{A}_1(J)}_a^c + \prs{\mathcal{A}_2 (J)}_a^c.
\end{split}
\end{align}
For all terms present, we will need to understand derivatives of $\vartheta$ up to highest order. With this in mind, observe that
\begin{align*}
\N_a \vartheta_b &=\tfrac{1}{2} \gw^{ji} \N_a \psi_{ijb}\\
&= \tfrac{1}{2} \gw^{ji} \N_a \prs{\del_i \gw_{jb} + \del_j \gw_{bi} + \del_b \gw_{ij}}\\
&=\gw^{ji} \N_a \prs{\del_i \gw_{jb} }+\tfrac{1}{2} \gw^{ji} \N_a \prs{\del_b \gw_{ij}}.
\end{align*}
Then, we have that
\begin{align}
\begin{split}\label{eq:Nvartheta}
\brk{\N_a \vartheta_b}_{2 \slash{g}} &= \brk{- \gw^{ji}g_{ju} \N_a \prs{\del_i J_b^u }+\tfrac{1}{2} \gw^{ji} \N_a \prs{\del_b \gw_{ij}}}_{2 \slash{g}}\\
&=  \brk{J_u^i \N_a \prs{\del_i J_b^u }-\tfrac{1}{2}J_u^i \N_a \prs{\del_b J_i^u }}_{2 \slash{g}}\\
&= \brk{-J_b^u  \N_a \prs{\del_i J_u^i }-\tfrac{1}{2}J_u^i \N_a \prs{\del_b J_i^u }}_{2 \slash{g}}.
\end{split}
\end{align}
Now with reference to \eqref{eq:Psidecomp}, for $\prs{\mathcal{A}_1(J)}$, we expand out
\begin{align*}
\prs{\mathcal{A}_1(J)}_a^c = 2\gw^{cb} \prs{\N_a \vartheta_b}^{2,0+0,2} &=  \gw^{cb} \prs{\N_a \vartheta_b - J_a^w J_b^v \N_w \vartheta_v} \\
&= \underset{\mathcal{A}_{11}(J)}{\underbrace{\gw^{cb} \prs{\N_a \vartheta_b}}} +  \underset{\mathcal{A}_{12}(J)}{\underbrace{- \gw^{cb} \prs{J_a^w J_b^v \N_w \vartheta_v}}}.
\end{align*}
We address each subterm. First we have that
\begin{align*}
\brk{\mathcal{A}_{11}(J)}_{2 / g} &=\brk{ - J^b_v g^{cv} \prs{\N_a \vartheta_b}}_{2 / g}\\
&=\brk{  - J^b_v g^{cv} \prs{ -J_b^u  \N_a \prs{\del_i J_u^i }- \tfrac{1}{2} J_u^i \N_a \prs{\del_b J_i^u}}}_{2 / g}\\
&=\brk{ -  g^{cv}  \N_a \prs{\del_i J_v^i }+ \tfrac{1}{2}  J^b_v g^{cv}J_u^i \N_a \prs{\del_b J_i^u}}_{2 / g}.
\end{align*}
Likewise we compute
\begin{align*}
\brk{\mathcal{A}_{12}(J)}_{2/g} &=\brk{  J^b_r g^{cr} \prs{J_a^w J_b^v \N_w \vartheta_v}}_{2 / g}\\
&=\brk{  - J_a^w g^{cv} \prs{\N_w \vartheta_v}}_{2 / g}\\
&=\brk{  - J_a^w g^{cv}  \prs{-J_v^u  \N_w \prs{\del_i J_u^i } - \tfrac{1}{2} J_u^i \N_w \prs{\del_v J_i^u}}}_{2 / g} \\
&= \brk{ J_a^w J_v^u g^{cv} \N_w \prs{\del_i J_u^i } +\tfrac{1}{2}  J_a^w g^{cv}  J_u^i \N_w \prs{\del_v J_i^u}}_{2 / g}.
\end{align*}
We now the linearisation of each and compute the off diagonal inner product with another $J$-variation, $H$. Keep in mind that for all $J$-variations, we have that $J_a^e H_e^b = - H_a^e J_e^b$.
\begin{align*}
\ip{\mathcal{L}_{\mathcal{A}_{11}}^{\xi}(K) , H} &=-  g^{cv} \prs{   \xi_a \xi_i K_v^i } g^{ar} H^e_r g_{ce}\\
&=- \prs{   \xi^r \xi_i K_v^i } H^v_r .\\
\ip{\mathcal{L}_{\mathcal{A}_{12}}^{\xi}(K) , H} &=  J_v^u J_a^w g^{cv}\prs{ \xi_w \xi_i K_u^i } g^{ar} H^e_r g_{ce} \\
&= - H_v^u \prs{ \xi^v \xi_i K_u^i }.
\end{align*}
Taking the sum we conclude that for $\mathcal{A}_1$,
\begin{align*}
\ip{\mathcal{L}_{\mathcal{A}_{1}}^{\xi}(K) , H} &=-2 H_v^u  \prs{ \xi^v \xi_i K_u^i }.
\end{align*}
Next we address term $\mathcal{A}_2$ of \eqref{eq:Psidecomp},
\begin{align*}
\mathcal{A}_{2}(J) =2 \gw^{cb} \prs{\N_b \vartheta_a}^{2,0+0,2} &= \gw^{cb} \prs{\N_b \vartheta_a - J_b^w J_a^v \N_w \vartheta_v} \\
&= \underset{\mathcal{A}_{21}(J)}{\underbrace{\gw^{cb} \prs{\N_b \vartheta_a}}} +  \underset{\mathcal{A}_{22}(J)}{\underbrace{-  \gw^{cb} \prs{J_b^w J_a^v \N_w \vartheta_v}}}.
\end{align*}
For the first term we compute
\begin{align*}
\brk{\mathcal{A}_{21}(J)}_{2 / g} &= \brk{\gw^{cb} \prs{\N_b \vartheta_a}}_{2 / g} \\
&=\brk{ - J^b_y g^{cy} \prs{-J_a^u  \N_b \prs{\del_i J_u^i } - \tfrac{1}{2} J_u^i \N_b \prs{\del_a J_i^u}}}_{2 / g}\\
&=\brk{ J^b_y g^{cy} J_a^u\prs{  \N_b \del_i J_u^i }+ \tfrac{1}{2} J^b_y g^{cy}  J_u^i \N_b \prs{\del_a J_i^u}}_{2 / g}.
\end{align*}
Next we have that
\begin{align*}
\brk{\mathcal{A}_{22}(J) }_{2 / g}&=\brk{ - \gw^{cb} \prs{J_b^w J_a^v \N_w \vartheta_v}}_{2 / g} \\
&=\brk{ - g^{wc}  J_a^v \prs{\N_w \vartheta_v}}_{2 / g}\\
&=\brk{- g^{wc}  J_a^v  \prs{ -J_v^u  \N_w \prs{\del_i J_u^i } - \tfrac{1}{2} J_u^i \N_w \del_v J_i^u}}_{2 / g}\\
&=\brk{ - g^{wc}   \prs{   \N_w \del_i J_a^i }+\tfrac{1}{2}  g^{wc}  J_a^v J_u^i \N_w \del_v J_i^u}_{2 / g}.
\end{align*}
Therefore it follows that
\begin{align*}
\ip{\mathcal{L}^{\xi}_{\mathcal{A}_{21}}(K), H} &=  J^b_y g^{cy} J_a^u\prs{  \xi_b \xi_i K_u^i } g_{cm} g^{an} H^m_n \\
&= - H^b_y  \prs{  \xi_b \xi_i K_u^i } g^{yu} .\\
\ip{\mathcal{L}_{\mathcal{A}_{22}}^{\xi}(K), H} &=  - g^{wc}   \prs{   \xi_w \xi_i K_a^i } g_{cm} g^{an} H^m_n \\
&=  - \prs{   \xi_m \xi_i K_a^i } g^{an} H^m_n.
\end{align*}
Taking the sum of the two, we have
\begin{align*}
\ip{\mathcal{L}^{\xi}_{\mathcal{A}_2}(K), H} &= - 2 H^b_y  \prs{  \xi_b \xi_i K_u^i } g^{yu}.
\end{align*}
Thus we conclude that, noting that $H g$ is skew symmetric for variations where $g$ is fixed,
\begin{align*}
\ip{\mathcal{L}^{\xi}_{\mathcal{A}}(K), H} &= -2 H^b_y  \prs{  \xi_b \xi_i K_u^i } g^{yu}- 2 H_v^u \prs{ \xi^v \xi_i K_u^i } \\
&= 0.
\end{align*}
The result follows.
\end{proof}
Since $\mathcal{A}$ is addressed, we next manipulate $\mathcal{B}$ of \eqref{eq:Psi} and observe that
\begin{align}\label{eq:Bmanip}
\begin{split}
\prs{\mathcal{B}(J)}_a^c &= 2 g^{jc} \prs{\N^e \psi_{jae} -  J_a^kJ_e^r\N^e \psi_{jkr}}^{2,0+0,2}\\
&= g^{jc} \N^e  \prs{\psi_{jae} -  J_a^kJ_e^r \psi_{jkr} - J_j^m J_a^k \psi_{mke} + J_j^m J_a^k J_k^hJ_e^r \psi_{mhr} }\\
&= g^{jc} \N^e  \prs{\psi_{jae} -  J_a^kJ_e^r \psi_{jkr} - J_j^m J_a^k \psi_{mke} - J_j^m J_e^r \psi_{mar} }\\
&= 4 g^{jc} \N^e  \psi_{jae}^{3,0+0,3}.
\end{split}
\end{align}
Since $\Lambda_3^{3,0+0,3}$ is trivial in dimension $4$, symbol of $\mathcal{B}$ is as well. However, when $\dim M > 4$, $\mathcal{B}(J)$ has nontrivial symbol and will destroy the parabolicity without the help of $\mho$. Observing that $\N^e  \psi_{jae}^{3,0+0,3}$ is skew-symmetric and J-antiinvariant, we will include this in our choice of $\mho$. Additionally, we will insert one final term whose presence will be justified in \S \ref{ss:mtproof}.
\begin{equation}\label{eq:mhocomplete}\mho_{av} \triangleq   \kappa_{av}  - 2\N^e  \psi_{vae}^{3,0+0,3} - \prs{\prs{\mathcal{L}_{\vartheta} J}^m_a g_{mv}}_{\text{skew}}.
\end{equation}
%

equivalently, we have
\begin{equation}\label{eq:Jflow}
\prs{\tfrac{\del J}{\del t}}_a^c \,= \,- \prs{\cP_{av} - 2 J_a^y \Rc_{yv}}^{2,0+0,2} g^{vc} + \prs{\kappa_{av}  - 2\N^e  \psi_{vae}^{3,0+0,3}}g^{cv} - \prs{\prs{\mathcal{L}_{\vartheta} J}^m_a g_{mv}}_{\text{skew}} g^{vc},
\end{equation}
where $\kappa \in TM_{\otimes 2}$ is a skew symmetric and $J$-antiinvariant function of $N$ and $\psi$, in particular, it is of lower order in both $g$ and $J$.
Summarizing the above, we see that the symbol of the above flow is given by that of 
\begin{equation}\label{eq:blah}-\tfrac{1}{2}  J_a^k \N^l N_{lk}^c - \prs{\prs{\mathcal{L}_{\vartheta} J}^m_a g_{mv}}_{\text{skew}} g^{vc},\end{equation}
which was computed in the almost Hermitian case in Proposition 5.4 of \cite{ST}.
\subsection{Proof of main theorem}\label{ss:mtproof}
We now look at the gauge modified version of the flow. First, letting $\overline{\N}$ be some background torsion-free connection, set
\begin{equation}\label{eq:Xdef}
Z^p \equiv \prs{Z(\gw,J, \overline{\N})}^p \triangleq \gw^{kl} \overline{\N}_k J_l^p = g^{kl} \prs{\overline{\gG}_{kl}^p - \gG_{kl}^p} + \vartheta^p,
\end{equation}
where the final equality follows by Lemma \ref{lem:vectorfield} in the appendix, where we improve the statement of (5.13) of \cite{ST}. In particular, up to addition by the Lie form, this vector field is the same used in the short time existence proof for Ricci flow, which we will set as
\begin{equation*}
X^p \equiv g^{kl} \prs{\gG_{kl}^p - \overline{\gG}_{kl}^p}.
\end{equation*}
%
%
Note in particular that $Z = - X + \vartheta$. We now consider the following gauge modification of \eqref{eq:AHCF}.
\begin{align}\label{eq:flowflow}
\begin{cases}
\prs{\tfrac{\del J}{\del t}}_{a}^c &=- \prs{\cP_{av} - 2 J_a^y \Rc_{yv}}^{2,0+0,2} g^{vc} + \prs{\kappa_{av}  - 2\N^e  \psi_{vae}^{3,0+0,3}}g^{cv} \\
&\qquad \qquad \qquad \qquad - \prs{\prs{\mathcal{L}_{\vartheta} J}^m_a g_{mv}}_{\text{skew}} g^{vc} + \prs{\mathcal{L}_{X(g,J)} J}_a^c \triangleq \mathcal{D}_1(g,J),\\
\prs{\tfrac{\del g}{\del t}}_{bc} &= - 2 \Rc_{bc} + \prs{\mathcal{L}_X(g,J) g}_{bc} \triangleq \mathcal{D}_2(g,J).
\end{cases}
\end{align}
Via \eqref{lem:A(J)class} and \eqref{eq:operator decomp}, up to highest order,
\begin{align}\label{eq:D2J}
\begin{split}
\brk{\prs{\mathcal{D}_1(g,J)}_a^c}_2 &= \brk{-\tfrac{1}{2}  J_a^k \N^l N_{lk}^c+ \tfrac{1}{2}\prs{\mathcal{A}(J)}_a^c + \tfrac{1}{2} \prs{\mathcal{B}(J)}_a^c + \prs{ - 2\N^e  \psi_{vae}^{3,0+0,3}}g^{cv} - \prs{\prs{\mathcal{L}_{\vartheta} J}^m_a g_{mv}}_{\text{skew}} g^{vc} + \prs{\mathcal{L}_{X(g,J)} J}_a^c}_2 \\
&= \brk{-\tfrac{1}{2}  J_a^k \N^l N_{lk}^c- \prs{\mathcal{L}_{Z(g,J)} J}_a^c+ \tfrac{1}{2}\prs{\mathcal{A}(J)}_a^c -\prs{\prs{\mathcal{L}_{\vartheta} J}^m_a g_{mv}}_{\text{skew}} g^{vc} + \prs{\mathcal{L}_{\vartheta(g,J)} J}_a^c}_2 \\
\end{split}
\end{align}
Referring to Lemma 3.3 of \cite{Dai}, translated in coordinates we have that
\begin{align*}
\prs{\mathcal{L}_{\vartheta}g}_{ij}^{2,0+0,2} = J_j^b \prs{\prs{\mathcal{L}_{\vartheta} J}_i^a g_{ab}}_{\text{sym}}.
\end{align*}
Thus, it follows that
\begin{align*}
\prs{\mathcal{L}_{\vartheta} J}_a^c &= \prs{\prs{\mathcal{L}_{\vartheta} J}^m_a g_{mv}} g^{vc} \\
&=\prs{\prs{\mathcal{L}_{\vartheta} J}^m_a g_{mv}}_{\text{sym}} g^{vc} +\prs{\prs{\mathcal{L}_{\vartheta} J}^m_a g_{mv}}_{\text{skew}} g^{vc}\\
&= - J_v^j \prs{\mathcal{L}_{\vartheta} g}_{aj}^{2,0+0,2} g^{vc} +\prs{\prs{\mathcal{L}_{\vartheta} J}^m_a g_{mv}}_{\text{skew}} g^{vc} \\
&= \gw^{cj} \prs{\mathcal{L}_{\vartheta} g}_{aj}^{2,0+0,2}+\prs{\prs{\mathcal{L}_{\vartheta} J}^m_a g_{mv}}_{\text{skew}} g^{vc}.
\end{align*}
Note that since $\prs{\mathcal{L}_{\vartheta} J}g$ is of type $(2,0)+(0,2)$, this means that the right side term is indeed of the same type (when lowered by the metric). It is also worth noting that $\prs{\prs{\mathcal{L}_{\vartheta} J}g}_{\text{skew}}$ corresponds with $\prs{\mathcal{L}_{\vartheta} \gw}^{2,0+0,2}$. This justifies the choice of $\mho$ in \eqref{eq:mhocomplete}.
\begin{align}\label{eq:D2J}
\begin{split}
\brk{\prs{\mathcal{D}_1(g,J)}_a^c}_2
&= \brk{-\tfrac{1}{2}  J_a^k \N^l N_{lk}^c- \prs{\mathcal{L}_{Z(g,J)} J}_a^c+ \prs{\mathcal{A}(J)}_a^c}_2.
\end{split}
\end{align}
By Proposition \ref{prop:dim4leeform} and \eqref{eq:blah}, combined with Proposition 5.4 of \cite{ST}, it follows that
\begin{align*}
\gs \brk{\widehat{\mathcal{L}_J \mathcal{D}_1} } (K)_{i}^j &= \brs{\xi}^2 K_{i}^j.
\end{align*}

Now we consider the next operator, $\mathcal{D}_2$.
\begin{align*}
\prs{\mathcal{D}_2(g,J)}_{bc} 
&= - 2 \Rc_{bc} + \prs{\mathcal{L}_{X(g,J)} g}_{bc}.
\end{align*}
Note that the Ricci tensor is completely independent of $J$ and there is only lower order dependence in the remnant Lie derivative $\mathcal{L}_X g$, and obviously has the correct symbol. In particular
\begin{align*}
\gs \brk{\widehat{\mathcal{L}_J \mathcal{D}_2} } (K)_{i}^j =0, \qquad
\gs \brk{\widehat{\mathcal{L}_g \mathcal{D}_2} } (h)_{ij} &= \brs{\xi}^2 h_{ij}.
\end{align*}
In particular, then we have that
\begin{align*}
\gs \brk{\widehat{\mathcal{L D}}}(K,h) &= \begin{pmatrix} \Id & \ast \\ 0 & \Id \end{pmatrix} \begin{pmatrix} K \\ h \end{pmatrix}.
\end{align*}
It follows that \eqref{eq:flowflow} is a strictly parabolic system of equations. Hence, the short time existence follows from the standard theory for parabolic flows. We can also follow arguments in the proof of Theorem 1.1 in \cite{ST} to complete the proof.
\begin{rmk}
We observe that by freezing the metric this flow is reduced to one on purely almost complex structures, of the form
\begin{align*}
\prs{\tfrac{\del J}{\del t}}_{a}^c = - \cP_{av}^{2,0+0,2} g^{vc} 
-\prs{\prs{\mathcal{L}_{\vartheta} J}^m_a g_{mv}}_{\text{skew}} g^{vc} +\prs{\kappa_{av} + 2 \N^e  \prs{d \gw}_{eav}^{3,0+0,3}}g^{vc} .
\end{align*}
This in fact is a \emph{parabolic} flow on almost complex structures, which we will investigate in future work.
\end{rmk}
\section{Appendix}\label{section:appendix}
\begin{lemma}\label{lem:vectorfield}
For $\prs{M^n, J ,\gw}$ almost Hermitian, the vector field defined in \eqref{eq:Xdef} is
\begin{align*}
Z^p \equiv g^{uk} \prs{\overline{\gG}_{ku}^p- \gG_{ku}^p} + \vartheta^p.
\end{align*}
\begin{proof} Starting with \eqref{eq:Xdef} we combute
\begin{align}
\begin{split}\label{eq:Xderivation}
Z^p &= \gw^{kl} \prs{\overline{\N}_k J_l^p} \\
&= \gw^{kl} \prs{\del_k J_l^p} -\gw^{kl} \overline{\gG}_{kl}^m J_m^p + \gw^{kl} \overline{\gG}_{km}^p J^m_l \\
&= \gw^{kl} \prs{ \N_k J_l^p + \gY_{kl}^u J_u^p - \gY_{ku}^p J_l^u} -\overline{\gG}_{kl}^m J_m^p \gw^{kl} +  \overline{\gG}_{km}^p g^{mk} \\
&=  \gw^{kl}\gY_{kl}^u J_u^p - \gY_{ku}^p g^{uk}+  \overline{\gG}_{km}^p g^{mk}.
\end{split}
\end{align}
Let's investigate these terms. For the first, by Lemma \ref{lem:ThetaPsi},
\begin{align*}
\gw^{kl}\gY_{kl}^u J_u^p &= \gw^{kl}\prs{\gG_{kl}^u - \Theta_{kl}^u}J_u^p \\
&= 0.
\end{align*}
For the next term
\begin{align*}
\gY^p_{ku} g^{uk} &= \prs{\gG^p_{ku} - \Theta_{ku}^p} g^{uk} \\
&= g^{uk} \gG^p_{ku} + \vartheta^p.
\end{align*}
Inserting these into \eqref{eq:Xderivation} yields the result.
\end{proof}
\end{lemma}
We now justify \eqref{eq:3formtypedecomp} above regarding the type decompositions of $TM_{\otimes 3}$ by determining the corresponding projections into $TM_{\otimes 3}^{1,2+2,1}$ and $TM_{\otimes 3}^{3,0+0,3}$.
\begin{lemma}For $(M^n, \gw, J)$ almost Hermitian,
\begin{align*}
F^{3,0+0,3}_{ijk} &\equiv \tfrac{1}{4} \prs{F_{ijk} -J_j^b J_k^c F_{ibc} - J_i^a J_k^c F_{ajc} - J_i^a J_j^b F_{abk}}, \\
F^{2,1+1,2}_{ijk} &\equiv \tfrac{3}{4}F_{ijk} + \tfrac{1}{4} \prs{J_j^b J_k^c F_{ibc} + J_i^a J_k^c F_{ajc} + J_i^a J_j^b F_{abk}}.
\end{align*}
\begin{proof}
For a vector field $X \in TM$, define the projections
\begin{align*}
\Pi_{1,0} X \triangleq \tfrac{1}{2} \prs{ X + iJ X }, \qquad \Pi_{0,1} X \triangleq \tfrac{1}{2} \prs{ X - iJX}.
\end{align*}
With this we have that
\begin{align*}
\prs{\Pi_{3,0} F }\prs{X,Y,Z} &= F \prs{\Pi_{1,0} X, \Pi_{1,0} Y, \Pi_{1,0} Z } \\
&= \tfrac{1}{8} F \prs{X,Y,Z} + i \tfrac{1}{8} F\prs{X,Y,JZ} + i \tfrac{1}{8} F \prs{X,JY,Z} + \prs{i}^2 \tfrac{1}{8} F \prs{X,JY,JZ} \\
&\hsp + i \tfrac{1}{8} F \prs{JX,Y,Z }+ \prs{i}^2 \tfrac{1}{8} F \prs{JX,Y,JZ} \\
&\hsp+ \prs{i}^2 \tfrac{1}{8} F \prs{JX,JY,Z} + \prs{i}^3 \tfrac{1}{8} F \prs{JX,JY,JZ} \\
&= \tfrac{1}{8} F \prs{X,Y,Z} + i \tfrac{1}{8} F\prs{X,Y,JZ} + i \tfrac{1}{8} F \prs{X,JY,Z} - \tfrac{1}{8} F \prs{X,JY,JZ} \\
&\hsp + i \tfrac{1}{8} F \prs{JX,Y,Z }- \tfrac{1}{8} F \prs{JX,Y,JZ} -\tfrac{1}{8} F \prs{JX,JY,Z} - i \tfrac{1}{8} F \prs{JX,JY,JZ}.
\end{align*}
Similarly, we have that
\begin{align*}
\prs{\Pi_{0,3} F }\prs{X,Y,Z} &= F \prs{\Pi_{0,1} X, \Pi_{0,1} Y, \Pi_{0,1} Z } \\
&= \tfrac{1}{8} F \prs{X,Y,Z} - i \tfrac{1}{8} F\prs{X,Y,JZ} - i \tfrac{1}{8} F \prs{X,JY,Z} + \prs{i}^2 \tfrac{1}{8} F \prs{X,JY,JZ} \\
&\hsp - i \tfrac{1}{8} F \prs{JX,Y,Z }+ \prs{i}^2 \tfrac{1}{8} F \prs{JX,Y,JZ} \\
&\hsp+ \prs{i}^2 \tfrac{1}{8} F \prs{JX,JY,Z} - \prs{i}^3 \tfrac{1}{8} F \prs{JX,JY,JZ} \\
&= \tfrac{1}{8} F \prs{X,Y,Z} - i \tfrac{1}{8} F\prs{X,Y,JZ} - i \tfrac{1}{8} F \prs{X,JY,Z} - \tfrac{1}{8} F \prs{X,JY,JZ} \\
&\hsp - i \tfrac{1}{8} F \prs{JX,Y,Z }-  \tfrac{1}{8} F \prs{JX,Y,JZ} - \tfrac{1}{8} F \prs{JX,JY,Z} + i \tfrac{1}{8} F \prs{JX,JY,JZ}.
\end{align*}
Therefore we have that
\begin{align*}
\prs{\Pi_{3,0} F + \Pi_{0,3} F} \prs{X,Y,Z} &= \tfrac{1}{4} \prs{ F \prs{X,Y,Z} - F \prs{X,JY,JZ} - F \prs{JX,Y,JZ} - F \prs{JX,JY,Z}}.
\end{align*}
Now we compute
\begin{align*}
\prs{\Pi_{2,1}F + \Pi_{1,2} F}\prs{X,Y,Z} &=  \prs{F \prs{X,Y,Z} - \prs{\Pi_{3,0} F + \Pi_{0,3} F}} \prs{X,Y,Z}\\
&= \tfrac{3}{4} F \prs{X,Y,Z} +  \tfrac{1}{4} \prs{F \prs{X,JY,JZ} + F \prs{JX,Y,JZ} + F \prs{JX,JY,Z}}.
\end{align*}
Converting to coordinates yields the result.
\end{proof}
\end{lemma}
\bibliography{sources.bib}{}
\bibliographystyle{alpha}
\end{document}